\documentclass{birkjour}
\usepackage[T1]{fontenc} 
\usepackage{latexsym} 
  \usepackage[all]{xy}
  \usepackage{amsfonts} 
  \usepackage{amsthm} 
  \usepackage{amsmath} 
  \usepackage{amssymb}
  \usepackage{pifont}  
  \usepackage{enumerate}
  
 \newtheorem{thm}{Theorem}[section]

 \theoremstyle{definition}
 
 \theoremstyle{remark}

 \numberwithin{equation}{section}
 
 \def\bs{\begin{statement}}
\def\es{\end{statement}}
  \swapnumbers
  \newtheorem{statement}[thm]{}

  \newcounter{zlist}

  \newcounter{blist}

  \newcounter{rlist}

\def\AA{{\mathbb A}}

\def\CC{{\mathbb C}}

\def\NN{{\mathbb N}}

\def\ZZ{{\mathbb Z}}

\def\tt{{\mathfrak T}}

\newcommand{\Cc}{\mathcal{C}}

\newcommand{\Oo}{\mathcal{O}}

\def\O{{\bf O}}

\def\*C{{}^*\hspace*{-1pt}{\Cc}}

\def\text#1{{\rm {\rm #1}}}

 \def\1{\mathbf{1}}

  \def\qn#1#2{\left[#1;#2\right]}

\begin{document}

%
%
%
%
%
%
%
%
%

\title[A curious differential calculus on the quantum disc and cones]
 {A curious differential calculus\\ on the quantum disc and cones}

\author{Tomasz Brzezi\'nski}

\address{%
Department of Mathematics, Swansea University, 
  Swansea SA2 8PP, U.K.\ \newline 
Department of Mathematics, University of Bia{\l}ystok, K.\ Cio{\l}kowskiego  1M,
15-245 Bia\-{\l}ys\-tok, Poland. E-mail: {T.Brzezinski@swansea.ac.uk} }


\author{Ludwik D\k{a}browski}
\address{SISSA, Via Bonomea 265, 34136 Trieste, Italy. E-mail: dabrow@sissa.it}



\date{October 2016}

\begin{abstract}
A non-classical differential calculus on the quantum disc and cones is constructed and the associated integral is calculated.
\end{abstract}

\maketitle

\section{Introduction}
The aim of this note is to present a two-dimensional differential calculus on the quantum disc algebra, which has no counterpart in the classical limit, but admits a well-defined (albeit different from the one in \cite{BegMaj:spe}) integral, and restricts properly to the quantum cone algebras. In this way the results of \cite{Brz:dif} are extended to other classes of non-commutative surfaces and to higher forms. The presented calculus is associated to an orthogonal pair of skew-derivations, which arise as a particular example of skew-derivations on generalized Weyl algebras constructed recently in \cite{AlmBrz:ske}. It is also a fundamental ingredient in the construction of the Dirac operator on the quantum cone \cite{BrzDab:spe} that admits a twisted real structure in the sense of \cite{BrzCic:twi}.

The reader unfamiliar with non-commutative differential geometry notions is referred to \cite{Brz:abs}.

\section{A differential calculus on the quantum disc}\label{sec.diff}
Let $0<q<1$. The coordinate algebra of the quantum disc, or the quantum disc algebra $\Oo(D_q)$ \cite{KliLes:two} is a complex $*$-algebra generated by $z$  subject to 
\begin{equation}\label{disc}
z^*z - q^2zz^* = 1-q^2.
\end{equation}
To describe the algebraic contents of $\Oo(D_q)$ it is convenient to introduce a self-adjoint element $x = 1-zz^*$, which $q^2$-commutes with the generator of $\Oo(D_q)$, $xz = q^2zx$. A linear basis of $\Oo(D_q)$ is given by monomials 
$x^kz^l, x^k{z^*}^{l}$. 
We 
view $\Oo(D_q)$ as a $\ZZ$-graded algebra, setting $\deg(z)=1$, $\deg(z^*)=-1$. Associated with this grading is the degree-counting automorphism
$\sigma:\Oo(D_q) \to \Oo(D_q)$, defined on homogeneous  $a\in \Oo(D_q)$ by $\sigma(a) = q^{2\deg(a)} a$. As explained in \cite{AlmBrz:ske} there is an orthogonal pair of skew-derivations $\partial,\bar\partial: \Oo(D_q)\to \Oo(D_q)$ twisted by $\sigma$ and given on the generators of $\Oo(D_q)$ by
\begin{equation}\label{partial}
\partial(z) = z^*, \quad \partial(z^*) = 0, \qquad \bar\partial(z) = 0, \quad \bar\partial(z^*) = q^2z,
\end{equation}
and extended to the whole of   $\Oo(D_q)$  by the (right) $\sigma$-twisted Leibniz rule. Therefore, there is also a corresponding first-order differential calculus $\Omega^1 (D_q)$ on $\Oo(D_q)$, defined as follows.

As a left $\Oo(D_q)$-module, $\Omega^1 (D_q)$ is freely generated by one forms $\omega, \bar\omega$. The right $\Oo(D_q)$-module structure and the differential $d :\Oo(D_q) \to \Omega^1 (D_q)$ are defined by
\begin{equation}\label{diff}
\omega a = \sigma(a) \omega, \quad \bar\omega a = \sigma(a) \bar\omega, \qquad d(a) = \partial(a)\omega + \bar\partial(a)\bar\omega.
\end{equation}
In particular,
\begin{equation}\label{dz}
dz = z^*\omega = q^2\omega z^*, \qquad dz^*= q^2z \bar\omega = \bar\omega z, 
\end{equation}
and so, by the commutation rules \eqref{diff},
\begin{equation}\label{omegadz}
\omega = \frac{q^{-2}}{1-q^2}\left(dz z - q^4zdz\right), \qquad \bar\omega = \frac{q^{-2}}{1-q^2}\left( z^*dz^* - q^2dz^*z^*\right).
\end{equation}
Hence $\Omega^1 (D_q) = \{\sum_i a_i db_i\; |\; a_i,b_i\in \Oo(D_q)\}$, i.e.\ $(\Omega^1 (D_q), d)$ is truly a first-order differential calculus not just a degree-one part of a differential graded algebra. The appearance of $q^2-1$  in the denominators in \eqref{omegadz} indicates that this  calculus has no classical (i.e.\ $q=1$) counterpart. 

The first-order calculus $(\Omega^1 (D_q), d)$ is a $*$-calculus in the sense that the $*$-structure extends to  the bimodule $\Omega^1 (D_q)$ so that $(a \nu b)^* = b^*\nu^* a^*$ and $(da)^* =d(a^*)$, for all $a,b\in \Oo(D_q)$ and $\nu \in \Omega^1 (D_q)$, provided $\omega^* = \bar\omega$ (this choice of the $*$-structure justifies the appearance of $q^2$ in the definition of $\bar\partial$ in equation \eqref{partial}). From now on we view $(\Omega^1 (D_q), d)$ as a $*$-calculus, which allows us to reduce by half the number of necessary checks.

Next we aim to show that the module of 2-forms  $\Omega^2 (D_q)$ obtained by the universal extension of  $\Omega^1 (D_q)$ is  generated by the anti-self-adjoint 2-form\footnote{One should remember that the $*$-conjugation takes into account the parity of the forms; see \cite{Wor:dif}.}
\begin{equation}\label{volume}
\mathsf{v} = \frac{q^{-6}}{q^2-1}(\omega^*\omega + q^8\omega\omega^*), \qquad \mathsf{v}^* = - \mathsf{v}
\end{equation}
and to describe the structure of $\Omega^2 (D_q)$. By \eqref{diff}, for all $a\in \Oo(D_q)$,
\begin{equation}\label{va}
 \mathsf{v} a = \sigma^2(a) \mathsf{v}.
 \end{equation} 
 Combining commutation rules \eqref{diff} with the relations \eqref{dz} we obtain
\begin{equation}\label{zdz}
z^* dz = q^2dzz^*, \qquad dzz - q^4 zdz = q^2(1-q^2)\omega,
\end{equation}
and their $*$-conjugates. The differentiation of the first of equations \eqref{zdz} together with \eqref{diff} and \eqref{disc} yield
\begin{equation}\label{omom}
\omega\omega^* =(1-x)\mathsf{v}, \qquad \omega^*\omega = q^6(q^2x -1)\mathsf{v},
\end{equation}
which means that $\omega\omega^*$ and $\omega^*\omega$ are in the module generated by $\mathsf{v}$. Next, by differentiating $\omega z^* = q^{-2}z^*\omega$ and $\omega z = q^2z\omega$ and using \eqref{dz} and \eqref{diff} one obtains
\begin{equation}\label{domegaz}
d\omega z^* = q^{-2} z^*d\omega + z(\omega^*\omega+q^4\omega\omega^*), \quad d\omega z = q^2zd\omega + (q^2+q^{-2})z^*\omega^2.
\end{equation}
The differentiation of $dz = z^* \omega$ yields 
\begin{equation}\label{zdom}
z^*d\omega = -q^2z\omega^*\omega.
\end{equation}
Multiplying this relation by $z$ from left and right, and using commutation rules \eqref{disc} and \eqref{diff} one finds that
$
(1-x) d\omega = q^{-4} z^*d\omega z.
$
Developing the right hand side of this equality with the help of the second of equations \eqref{domegaz} we find
\begin{equation}\label{domega}
d\omega =  \frac{1+q^{-4}}{q^2-1} {z^*}^2 \omega^2.
\end{equation}
Combining \eqref{domegaz} with \eqref{domega} we can derive
\begin{equation}\label{zomega}
{z^*}^3\omega^2 = -z\frac{q^8}{q^4 +1}\left(\omega^*\omega +q^4\omega\omega^*\right) .
\end{equation}
The multiplication of \eqref{zomega} by $z^3$ from the left and right  and the usage of \eqref{disc}, \eqref{diff} give
\begin{subequations}\label{omegas}
\begin{equation}\label{omega1}
(1-x)(1-q^{-2}x)(1-q^{-4}x) \omega^2 = - \frac{q^8}{q^4 +1}z^4\left(\omega^*\omega +q^4\omega\omega^*\right) ,
\end{equation}
\begin{equation}\label{omega2}
(1-q^2x)(1-q^{4}x)(1-q^{6}x) \omega^2 = - \frac{q^8}{q^4 +1}z^4\left(\omega^*\omega +q^4\omega\omega^*\right) .
\end{equation}
\end{subequations}
 Comparing the left hand sides of equations \eqref{omegas}, we conclude  that 
\begin{equation}\label{xomegas}
x\omega^2 = 0 = \omega^2 x \quad \mbox{and, by $*$-conjugation,} \quad x\omega^{*2} = 0 = \omega^{*2} x,
\end{equation}
and hence in view of either of \eqref{omegas}
\begin{equation}\label{omega.sq}
\omega^2 = - \frac{q^8}{q^4 +1}z^4\left(\omega^*\omega +q^4\omega\omega^*\right).
\end{equation}
By \eqref{omom}, the right hand side of \eqref{omega.sq} is in the module generated by $\mathsf{v}$, and so is $\omega^2$ and its adjoint $\omega^{*2}$. Thus, the module $\Omega^2 (D_q)$ spanned by all products of pairs of one-forms is indeed generated by $\mathsf{v}$. 

Multiplying \eqref{domega} and \eqref{zdom} by $x$ and using relations \eqref{xomegas} we obtain 
\begin{equation}\label{xzomegas1}
xz\omega^*\omega = 0 = \omega^*\omega xz.
\end{equation}
Following the same steps but now starting with the differentiation of 
$dz^*\!=\!q^2z \omega^*$ (see \eqref{dz}), we obtain the complementary relation
\begin{equation}\label{xzomegas2}
xz\omega\omega^* = 0 = \omega\omega^* xz.
\end{equation}
In view of the definition of $\mathsf{v}$, \eqref{xzomegas1} and  \eqref{xzomegas2} yield
$
xz\mathsf{v} = 0 = \mathsf{v}xz.
$
Next, the multiplication of, say, the first of  these equations from the left and right by $z^*$ and the use of \eqref{disc} yield
$
x(1-x)\mathsf{v} =0$  and $x(1-q^2x)\mathsf{v} =0$. 
The subtraction of one of these equations from the suitable scalar multiple of the other produces the necessary relation
\begin{equation}\label{xv}
x\mathsf{v} = 0 = \mathsf{v} x,
\end{equation}
which fully characterises the structure of $\Omega^2 (D_q)$ as an $\Oo(D_q)$-module generated by $\mathsf{v}$. In the light of \eqref{xv}, the $\CC$-basis of $\Omega^2 (D_q)$ consists of elements $\mathsf{v}z^n$, $ \mathsf{v}z^{*m}$, and hence, for all $w\in \Omega^2 (D_q)$, $wx = xw =0$, i.e.,  $\Omega^2 (D_q)$  is a torsion (as a left and right $\Oo(D_q)$-module). Since $\Oo(D_q)$ is a domain and $\Omega^2 (D_q)$ is a torsion, the dual of  $\Omega^2 (D_q)$ is the zero module, hence, in particular $\Omega^2 (D_q)$ is not projective.  Again by \eqref{xv}, the annihilator  of $\Omega^2 (D_q)$, 
$$
\mathrm{Ann}(\Omega^2 (D_q)) := \{ a\in \Oo(D_q) \; |\; \forall  w\in \Omega^2 (D_q), \, aw=wa=0\},
$$
is the ideal of $\Oo(D_q)$ generated by $x$.  The quotient $\Oo(D_q)/ \mathrm{Ann}(\Omega^2 (D_q))$ is the Laurent polynomial ring in one variable, i.e.\ the algebra $\Oo(S^1)$ of coordinate functions on the circle. When viewed as a module over $\Oo(S^1)$, $\Omega^2 (D_q)$ is free of rank one, generated by $\mathsf{v}$. Thus, although the module of 2-forms over $\Oo(D_q)$ is neither free nor projective, it can be identified with sections of a trivial line bundle once pulled back to the (classical) boundary of the quantum disc.

With \eqref{xv} at hand, equations \eqref{omom}, \eqref{omega.sq}, \eqref{domega} and their $*$-conjugates give the following relations in $\Omega^2 (D_q)$
\begin{subequations}\label{full}
\begin{equation}
d\omega = q^8 z^2\mathsf{v}, \quad d\omega^* = -z^{*2}\mathsf{v}, \quad \omega\omega^* = \mathsf{v}, \quad \omega^*\omega = -q^6\mathsf{v}, 
\end{equation}
\begin{equation}
\omega^2 = q^{12}\frac{q^2 -1}{q^4+1}z^4\mathsf{v}, \qquad  \omega^{*2} = q^{-4}\frac{q^2 -1}{q^4+1}z^{*4}\mathsf{v}.
\end{equation}
\end{subequations}
One can easily check that \eqref{full}, \eqref{xv} and \eqref{va} are consistent with \eqref{diff} with no further restrictions on $\mathsf{v}$.
Setting $\Omega^n (D_q) =0$, for all $n>2$, we thus obtain a 2-dimensional calculus on the quantum disc.

\section{Differential calculus on the quantum cone}
The quantum cone algebra $\Oo(C^N_q)$ is a subalgebra of $\Oo(D_q)$ consisting of all elements of the $\ZZ$-degree congruent to 0 modulo a positive natural number $N$. Obviously $\Oo(C^1_q)= \Oo(D_q)$, the case we dealt with in the preceding section, so we may assume $N>1$. $\Oo(C^N_q)$ is a $*$-algebra generated by the self-adjoint $x=1-zz^*$ and by $y=z^N$, which satisfy the following commutation rules
\begin{equation}\label{cone}
xy = q^{2N} yx, \qquad yy^* = \prod_{l=0}^{N-1}\left(1-q^{-2l}x\right), \qquad y^*y = \prod_{l=1}^{N}\left(1-q^{2l}x\right).\end{equation}
The calculus  $\Omega(C^N_q)$ on $\Oo(C^N_q)$ is obtained by restricting of the calculus  $\Omega(D_q)$, i.e.\  $\Omega^n(C^N_q) = \{ \sum_{i}a_0^i d(a_1^i)\cdots d(a_n^i)a_{n+1}^i\; |\; a_k^i\in \Oo(C^N_q)\}$. Since $d$ is a degree-zero map $\Omega(C^N_q)$ contains only these forms in $\Omega(D_q)$, whose $\ZZ$-degree is a multiple of $N$. We will show that all such forms are in $\Omega(C^N_q)$. Since  $\deg(\omega)= 2$,  $\deg(\omega^*) =-2$ and  $\deg(\mathsf{v})=0$,
this is equivalent to
$$
\Omega^1(C^N_q) = \Oo(D_q)_{\overline{-2}}\, \omega \oplus \Oo(D_q)_{\overline{2}}\, \omega^*, \qquad \Omega^2(C^N_q)= \Oo(C^N_q) \mathsf{v},
$$
where $\Oo(D_q)_{\overline{s}} = \{a\in \Oo(D_q)\; |\; \deg(a) \equiv s \!\! \mod\! N\}$.

As an $\Oo(C^N_q)$-module, $\Oo(D_q)_{\overline{-2}}$ is generated by $z^{N-2}$ and ${z^*}^2$, hence to show that  $\Oo(D_q)_{\overline{-2}}\, \omega\subseteq \Omega^1(C^N_q)$ suffices it to prove that $z^{N-2}\omega, {z^*}^2\omega \in \Omega^1(C^N_q)$. Using the Leibniz rule one easily finds that
$$
d y = \left(\qn{N}{q^2} -q^{-2N+4} \qn{N}{q^4}x\right) z^{N-2}\omega,
$$
where
$\qn{n}{s} := \frac{ s^n -1}{s-1}$. Hence, in view of \eqref{disc} and \eqref{diff},
\begin{subequations}\label{ydy}
\begin{equation}\label{y*dy}
y^* dy = \qn{N}{q^2}\left( 1 - q^4\frac{\qn{N}{q^4}}{\qn{N}{q^2}} x\right) \prod_{l=3}^N\left(1-q^{2l}x\right) {z^*}^2\omega,
\end{equation}
\begin{equation}\label{dyy*}
dy y^* = q^{-2N}\qn{N}{q^2}\left( 1 - q^{-2N+4}\frac{\qn{N}{q^4}}{\qn{N}{q^2}} x\right) \prod_{l=0}^{N-3}\left(1-q^{-2l}x\right) {z^*}^2\omega.
\end{equation}
\end{subequations}
The polynomial in $x$ on the right hand side of \eqref{y*dy} has  roots in common with the polynomial on the right hand side of \eqref{dyy*} if and only if there exists an integer
$k\in \left[-2N+2,  -N-1\right] \cup \left[2, N-1\right]$ 
such that 
\begin{equation}\label{crit}
q^{2k}(q^{2N} +1) = {q^2}+1.
\end{equation}
Equation \eqref{crit} is equivalent to
$
{q^2} \qn{N+k-1}{q^2} + \qn{k}{q^2} =0,
$
with the left hand side strictly positive if $k>0$ and strictly negative if $k\leq -N$. So, there are no solutions within the required range of values of $k$. 
Hence the polynomials  \eqref{y*dy}, \eqref{dyy*} are coprime, and so there exists a polynomial (in $x$) combination of the left had sides of  equations \eqref{ydy} that gives ${z^*}^2\omega$. This combination is an element of  $\Omega^1(C^N_q)$ and so is ${z^*}^2\omega$.
Next,
$$
{z^*}^2\omega \, y = q^{2N}(1-{q^2}x)(1-q^4 x)z^{N-2}\omega ,
$$
$$
 y  {z^*}^2\omega = (1-q^{-2N+4}x)(1-q^{-2N+2}x)z^{N-2}\omega,
$$
so again there is an $x$-polynomial combination of the left hand sides (which are already in $\Omega^1(C^N_q)$) giving $z^{N-2}\omega$. Therefore, $\Oo(D_q)_{\overline{-2}}\, \omega\subseteq \Omega^1(C^N_q)$. The case of $\Oo(D_q)_{\overline{2}}$ follows by the $*$-conjugation.

Since $z^2\omega^*$, ${z^*}^2\omega$
are elements of $\Omega^1(C^N_q)$,
\begin{equation}\label{o*o1}
\Omega^2(C^N_q) \ni z^2\omega^*{z^*}^2\omega = q^{-4}(1-x)(1-q^{-2}x) \omega^*\omega = -q^2 \mathsf{v},
\end{equation}
by the quantum disc relations and \eqref{full} and \eqref{xv}. 
Consequently, $\mathsf{v} \in \Omega^2(C^N_q)$. Therefore, $\Omega(C^N_q)$ can be identified with the subspace of $\Omega(D_q)$, of all the elements whose $\ZZ$-degree is a multiple of $N$.

\section{The integral}\label{sec.int}
Here we construct an algebraic integral associated to the calculus constructed in Section~\ref{sec.diff}. 
We start by observing that since $\sigma$ preserves the $\ZZ$-degrees of elements of $\Oo(D_q)$ and $\partial$ and $\bar\partial$ satisfy the $\sigma$-twisted Leibniz rules, the definition \eqref{partial} implies that $\partial$ lowers while $\bar\partial$ raises degrees by $2$. Hence, one can equip  $\Omega^1 (D_q)$  with the $\ZZ$-grading so that $d$ is the degree zero map, provided $\deg(\omega) = 2$,  $\deg(\omega^*) = -2$. Furthermore, in view of the definition of $\sigma$, one easily finds that
\begin{equation}\label{q-deriv}
\sigma^{-1}\circ\partial \circ\sigma = q^4 \partial, \qquad \sigma^{-1}\circ\bar\partial \circ\sigma = q^{-4} \bar \partial,
\end{equation}
i.e.\ $\partial$ is a $q^4$-derivation and $\bar\partial$ is a $q^{-4}$-derivation. Therefore, by \cite{BrzElK:int}, $\Omega (D_q)$ admits a divergence, for all right $\Oo(D_q)$-linear maps $f:\Omega^1 (D_q)\to \Oo(D_q)$, given by
\begin{equation}\label{div}
\nabla_0(f) = q^4\partial\left(f\left(\omega\right)\right) + q^{-4}\bar\partial\left(f\left(\omega^*\right)\right) .
\end{equation}
Since the $\Oo(D_q)$-module $\Omega^2 (D_q)$ has a trivial dual, $\nabla_0$ is flat. Recall that by the {\em integral} associated to $\nabla_0$ we understand the cokernel map of $\nabla_0$.
\begin{thm}\label{thm.int}
The integral associated to the divergence \eqref{div} is a map $\Lambda: \Oo(D_q) \to \CC$, given by
\begin{equation}\label{int}
\Lambda(x^kz^l) = \lambda \frac{\qn{k+1}{q^2}}{\qn{k+1}{q^4}} \delta_{l,0}, \qquad \mbox{for all $k\in \NN,\, l\in \ZZ$},
\end{equation}
where, for $l<0$,  $z^l$ means ${z^*}^{-l}$ and $\lambda \in \CC$.
\end{thm}
\begin{proof}
First we need to calculate the image of $\nabla_0$. Using the twisted Leibniz rule and the quantum disc algebra commutation rules \eqref{disc}, one obtains 
\begin{equation}\label{part1}
\partial(x^k) = -q^{-2}\qn{k}{q^4} x^{k-1}{z^*}^2.
\end{equation}
Since $\partial(z^*) =0$, \eqref{part1} means that all monomials $x^k {z^*}^{l+2}$ are in the image of $\partial$ hence in the image of $\nabla_0$. Using the $*$-conjugation we conclude the $x^k {z}^{l+2}$ are in the image of $\bar\partial$ hence in the image of $\nabla_0$. So $\Lambda$ vanishes on (linear combinations of) all such polynomials.
Next note that
\begin{equation}\label{part2}
\partial(z^2) = ({q^2}+1) - (q^4+1)x,
\end{equation}
hence
$$
\partial(z^*z^2 - q^4z^2z^*) = (1-q^4)z^*, \quad \partial(z^*z^2 - {q^2}z^2z^*) =(1-{q^2})(1+q^4)xz^*.
$$
This means that $z^*$ and $xz^*$ are in the image of $\partial$, hence of  $\nabla_0$. In fact, all the $x^kz^*$ are in this image which can be shown inductively. Assume $x^kz^* \in \mathrm{Im}(\partial)$, for all $k\leq n$. Then using the twisted Leibniz rule, \eqref{part1} and \eqref{part2} one finds
\begin{equation}\label{part3}
\partial(x^nz^2) = -{q^2}\qn{N}{q^4} x^{n-1} + ({q^2}+1)\qn{n+1}{q^4} x^n - \qn{n+2}{q^4}x^{n+1}.
\end{equation}
Since $\partial(z^*) =0$, equation \eqref{part3} implies that $\partial(z^nz^2z^*)$ is a linear combination of monomials $x^{n-1}z^*$, $x^nz^*$ and $x^{n+1}z^*$. Since the first two are in the image of $\partial$ by the inductive assumption, so is the third one. Therefore, all linear combinations of $x^kz^*$ and $x^kz$ (by the $*$-conjugation) are in the image of $\nabla_0$. 

Put together all this means that $\Lambda$ vanishes on all the polynomials
$\sum_{k,l=1}^n(c_{kl}x^kz^l + c'_{kl}x^k{z^*}^l).
$
The rest of the formula \eqref{int}  can be proven by induction. Set $\lambda = \Lambda(1)$. Since $\Lambda$ vanishes on all elements in the image of $\nabla_0$, hence also in the image of $\partial$, the application of $\Lambda$ to the right hand side of \eqref{part1} confirms \eqref{int} for $k=1$. Now assume that \eqref{int} is true for all $k\leq n$. Then the  application $\Lambda$ to the right hand side of \eqref{part3} followed by the use of the inductive assumption yields
\begin{eqnarray*}
\qn{n+2}{q^4}\Lambda\left(x^{n+1}\right) &=&  {q^2}\qn{N}{q^4} \Lambda\left(x^{n-1}\right) -({q^2}+1)\qn{n+1}{q^4} \Lambda\left(x^n\right) \\
&=& \lambda\left( ({q^2}+1)\qn{n+1}{q^2}-{q^2}\qn n{q^2}   \right) = \lambda\qn{n+2}{q^2}.
\end{eqnarray*}
Therefore, the formula \eqref{int} is true also for $n+1$, as required.
\end{proof}

The restriction of $\Lambda$ to the elements of $\Oo(D_q)$, whose $\ZZ$-degree is a multiple of $N$ gives an integral on the quantum cone $\Oo(C^N_q)$.


\subsection*{Acknowledgment}
The work on this project began during the first author's visit to SISSA, supported by INdAM-GNFM. He would like to thank the members of SISSA for hospitality. The second author was supported in part by the Simons Foundation grant 346300 and the Polish Government MNiSW 2015-2019 matching fund.

\end{document}